\documentclass[11pt]{amsart}

\usepackage{mathptmx} 
\usepackage[scaled=0.90]{helvet} 
\usepackage{courier} 
\normalfont
\usepackage[T1]{fontenc}
\usepackage{color}

\usepackage{hyperref}
\usepackage{amsmath}
\usepackage{amsthm}
\usepackage{amsfonts}
\usepackage{amssymb}
\usepackage{graphicx}

\newtheorem{theorem}{Theorem}[section]
\newtheorem{utv*}{Proposition}
\newtheorem{hyp*}{Conjecture}
\newtheorem{lemma}[theorem]{Lemma}
\newtheorem{corollary}[theorem]{Corollary}
\newtheorem{defin}{Definition}

\newtheorem*{th*}{Theorem}

\usepackage{hyperref}
\usepackage{amsmath}
\usepackage{amsthm}
\usepackage{amsfonts}
\usepackage{amssymb}
\usepackage{esint}
\usepackage{color}
\usepackage{hyperref}
\usepackage{graphicx}
\usepackage[left=0.7in,top=0.6in,right=0.5in,bottom=0.7in]{geometry}
\hypersetup{bookmarksdepth=3}

\numberwithin{equation}{section}

\newcommand{\ci}[1]{_{ {}_{\scriptstyle #1}}}

\newcommand{\f}{\varphi}
\newcommand{\R}{\mathbb{R}}

\newcommand{\Om}{\Omega}

\newcommand{\M}{\mathcal{M}}

\renewcommand{\d}{\delta}

\newcommand{\cz}{Calder\'{o}n--Zygmund\ }
\newcommand{\la}{\lambda}

\newcommand{\av}[2]{\left\langle #1\right\rangle_{_{\scriptstyle #2}}}

\def\cyr{\fontencoding{OT2}\fontfamily{wncyr}\selectfont}
\DeclareTextFontCommand{\textcyr}{\cyr}

\newcounter{vremennyj}

\newcommand{\ili}{\int\limits}
\newcommand{\sli}{\sum\limits}
\newcommand{\ve}{\varepsilon}
\newcommand{\ep}{\varepsilon}
\newcommand{\B}{\mathcal{B}}
\begin{document}

\title{Extremizers and Bellman function for martingale weak type inequality}

\author{Alexander Reznikov}
\address{Department of Mathematics,  Michigan State University, East
Lansing, MI 48824, USA}
\email{reznikov@ymail.com}

\author{Vasily Vasyunin}
\address{V. A. Steklov Math. Inst., Russian Academy of Science}
\email{vasyunin@pdmi.ras.ru}
\thanks{Work of V.~Vasyunin is supported  by the grant}

\author{Alexander Volberg}
\thanks{Work of A.~Volberg is supported  by the NSF under the grants  DMS-0758552, DMS-1265549
}
\address{Department of Mathematics, Michigan State University, East
Lansing, MI 48824, USA}
\email{volberg@math.msu.edu}
\urladdr{http://sashavolberg.wordpress.com}


\makeatletter
\@namedef{subjclassname@2010}{
  \textup{2010} Mathematics Subject Classification}
\makeatother

\subjclass[2010]{42B20, 42B35, 47A30}



%
%

\keywords{\cz operators, martingale transform, weak type estimates, Bellman function, extremal functions, stopping time}

   \begin{abstract}
We give an exact formula for the Bellman function of the weak type of martingale transform. We also give the extremal functions
(actually extremal sequences of functions). We find them using the precise form of the Bellman function. The extremal examples have a fractal nature as it often happens in that kind of problems. This article is devoted to the unweighted weak type estimate.\end{abstract}

\date{}
\maketitle


\section{Introduction}
In any harmonic analysis course it is proved that a Hilbert Transform $H$ satisfies the following weak $(1,1)$ inequality:
\begin{equation}\label{weakweak}
\left|\left\{x\colon |Hf(x)|\geqslant \la \right\}\right| \leqslant C\frac{\|f\|_1}{\la}, \;\;\; \forall \la>0, \;\;\; \forall f\in L^1.
\end{equation}
Here $|\cdot |$ denotes Lebesgue measure.
This inequality is proved by means of a famous Calder\'on-Zygmund decomposition of the function $f$. 
In this paper we present an alternative proof of \eqref{weakweak} for operators
$$
T_\ep f(x)=\sli_{I\subset I_0, \; I\in D} \ep_I (\varphi, h_I)h_I(x), \;\;\; |\ep_I|=1 \;\; \forall I\in D.
$$
instead of $H$. Here $I_0:=[0,1]$. $D$ denotes the dyadic lattice. $h\ci I$ are normalized in $L^2(\R)$ Haar function of
the cube (interval) $I$
$$
h\ci I(t):=
\begin{cases}
\phantom{-}\frac1{\sqrt{|I|}}\,,\quad& t\in I_+
\\
-\frac1{\sqrt{|I|}}\,,& t\in I_-
\end{cases}.
$$
Here $I_\pm$ are two halves of the
interval $I$.

\section{Bellman function}
Introduce a function 
$$
\mathcal{B}_0(\la, f, F) = \sup \left| \left\{x\colon \sli_{I\subset I_0, \; I\in D} \ep_I (\varphi, h_I)h_I(x) \geqslant \la \right\}\right|,
$$
where the supremum is taken over all families $\{\ep_I\}$ such that $|\ep_I|=1$, and all functions $\varphi$ with $\av{|\varphi|}{I_0}=F$, $\av{\varphi}{I_0}=f$.

Let $\Omega_0 = \{(\la, f, F)\colon F\geqslant |f|\}$ be the domain of $\mathcal{B}_0$.

Denote
$$
B_0(\la, f, F)=\begin{cases} 1, &\la \leqslant F \\
1-\frac{(\la-F)^2}{\la^2-f^2}, &\la\geqslant F, \end{cases} \; \; \;\;\;\;\;\; (F,f,\la)\in \Omega_0.
$$
Our main theorem is the following.
\begin{theorem}
For any $(\la, f, F)\in \Omega_0$ it holds that $\mathcal{B}_0(F, f, \la)=B_0(F, f, \la)$.
\end{theorem}

Firstly, it will be more convenient to work with a slightly modified function. We need a definition.
\begin{defin}
A function $\psi$ is called a martingale transform of a function $\varphi$, if for some family $\{\ep_I\}$, with $|\ep_I|=1$, 
$$
\psi(x)=\av{\psi}{I_0} + \sli_{I\subset I_0, \; I\in D} \ep_I (\varphi, h_I)h_I(x), \;\;\; x\in I_0.
$$
\end{defin}
Denote
$$
\mathcal{B}(g, f, F) = \sup \left| \left\{x\colon \psi(x)\geqslant 0 \right\}\right|,
$$
where the supremum is taken over all functions $\varphi$  with $\av{|\varphi|}{I_0}=F$, $\av{\varphi}{I_0}=f$, and all martingale transforms $\psi$ of $\varphi$ with $\av{\psi}{I_0}=g$.
It is easy to see that 
$$
\mathcal{B}_0(\la, f, F)=\mathcal{B}(-g, f, F).
$$
Denote $\Omega = \{(g, f, F)\colon F\geqslant |f|\}$ and
$$
B(g, f, F)=\begin{cases} 1, &-g \leqslant F \\
1-\frac{(g+F)^2}{g^2-f^2}, &-g\geqslant F, \end{cases} \; \; \;\;\;\;\;\; (g,f,F)\in \Omega.
$$
Then our main theorem is equivalent to the following one.
\begin{theorem}
For any $(g, f, F)\in \Omega$ it holds $\mathcal{B}(g, f, F)=B(g, f, F)$.
\end{theorem}
\begin{corollary}\label{cor1}
For any function $\varphi\in L^1$, any number $\la\geqslant 0$ and any family $\{\ep_I\}$ with $|\ep_I|=1$ it holds
$$
\left| \left\{x\colon \sli_{I\subset I_0, \; I\in D} \ep_I (\varphi, h_I)h_I(x) \geqslant \la \right\}\right| \leqslant 2\frac{\|\varphi\|_1}{\la}
$$
\end{corollary}
\begin{proof}
It is easy to verify that 
$$
\sup \left( \mathcal{B}_0(\la, f, F) \cdot \frac{\la}{F} \right) = 2.
$$
Thus, 
$$
\left| \left\{x\colon \sli_{I\subset I_0, \; I\in D} \ep_I (\varphi, h_I)h_I(x) \geqslant \la \right\}\right| \leqslant 2 \frac{F}{\la} = 2\frac{\|\varphi\|_1}{\la}.
$$
\end{proof}
\begin{corollary}
For any function $\varphi\in L^1$, any number $\la\geqslant 0$ and any family $\{\ep_I\}$ with $|\ep_I|=1$ it holds
$$
\left| \left\{x\colon \sli_{I\subset I_0, \; I\in D} \ep_I (\varphi, h_I)h_I(x) \geqslant \la \right\}\right| \leqslant 4\frac{\|\varphi\|_1}{\la}
$$
\end{corollary}
\begin{proof}
\begin{multline}\label{multtt}
\left| \left\{x\colon \sli_{I\subset I_0, \; I\in D} |\ep_I (\varphi, h_I)h_I(x)| \geqslant \la \right\}\right| = 
\left| \left\{x\colon \sli_{I\subset I_0, \; I\in D} \ep_I (\varphi, h_I)h_I(x) \geqslant \la \right\}\right| + \\
\left| \left\{x\colon \sli_{I\subset I_0, \; I\in D} \ep_I (-\varphi, h_I)h_I(x) \geqslant \la \right\}\right|  \leqslant 4\frac{\|\varphi\|_1}{\la}
\end{multline}
$$
$$
\end{proof}

We start to prove our main theorem.
\section{$B\geqslant \mathcal{B}$}
We need a technical lemma.
\begin{lemma}
Let $x^\pm$ be two points in $\Omega$ such tat
$|f^+-f^-|=|g^+-g^-|$ and $x=\frac12(x^++x^-)$. Then
\begin{equation}
\label{mi} B(x)-\frac{B(x^+)+B(x^-)}2\ge 0\,.
\end{equation}
\end{lemma}
Given the lemma, we prove the following theorem.
\begin{theorem}
For any point $x\in \Om$ it holds $B(x)\geqslant \mathcal{B}(x)$.
\end{theorem}
\begin{proof}
Let us fix a point $x\in\Omega$ and a pair of admissible functions $\f$, $\psi$
on $I_0$ corresponding to $x$. For any $I\in D$ by the symbol $x^{I}$ we denote
the point $(\av{\psi}I, \av{\varphi}I, \av{|\varphi|}I,)$. 
We notice that since $\psi$ is a martingale transform of $\varphi$, we always have 
$$
|f^{I^+}-f^{I^-}|=|g^{I^+}-g^{I^-}|, 
$$
and 
$$
x^{I}=\frac{x^{I^+}+x^{I^-}}2.
$$
Using consequently main
inequality for the function $B$ we can write down the following chain of
inequalities
$$
B(x)\ge\frac12\big(B(x^{I_0^+})+B(x^{I_0^-})\big)\ge\sum_{I\in
D,\,|I|=2^{-n}}\frac1{|I|}B(x^{I})=\int_0^1B(x^{(n)}(t))dt\,,
$$
where $x^{(n)}(t)=x^I$, if $t\in I$, $|I|=2^{-n}$.

Note that $x^{(n)}(t)\to(\psi(t), \varphi(t),|\varphi(t)|)$ almost everywhere (at any
Lebesgue point $t$), and therefore, since $B$ is continuous and bounded, we can
pass to the limit in the integral. So, we come to the inequality
\begin{equation}
\label{upest}
B(x)\ge\int_0^1B(\psi(t), \varphi(t),|\varphi(t)|)dt\ge\int_{\{t\colon\psi(t)\ge0\}}
=\big|\{t\in I_0\colon\psi(t)\ge0\}\big|\,
\end{equation}
where we have used the property $B(g,f,|f|)=1$ for $g\ge0$. Now, taking
supremum in~\eqref{upest} over all admissible pairs $\f$, $\psi$, we get the
required estimate $B(x)\ge\mathcal{B}(x)$.
\end{proof}
\section{$B(g, f, F)\leqslant \mathcal{B}(g, f, F)$}
This section is devoted to the following theorem.
\begin{theorem}
For any point $x\in \Om$ it holds $B(x)\leqslant \mathcal{B}(x)$.
\end{theorem}
To prove the theorem we need to present two sequences of functions $\{\varphi_n\}$, $\{\psi_n\}$, such that
\begin{itemize}
\item For every $n$ the function $\psi_n$ is a martingale transform of $\varphi_n$; 
\item For every $n$: $\av{|\varphi_n|}{I_0}=F$, $\av{\varphi_n}{I_0}=f$, $\av{\psi_n}{I_0}=g$;
\item It holds that $B(g, f, F)=\lim\limits_{n\to \infty} |\{x\colon \psi_n(x)\geqslant 0\}|$ 	.
\end{itemize}
We need the following definition.
\begin{defin}
We call a pair $(\varphi, \psi)$ admissible for the point $(g, f, F)$ if $\psi$ is a martingale transform of $\varphi$, and $\av{|\varphi|}{I_0}=F$, $\av{\varphi}{I_0}=f$, $\av{\psi}{I_0}=g$.
\end{defin}
\begin{defin}
We call a pair $(\varphi, \psi)$ an $\ep$-extremizer for a point $(g, f, F)$, if this pair is admissible for this point and $|\{x\colon \psi(x)\geqslant 0\}| \geqslant B(g, f, F)-\ep$.
\end{defin}
The following lemma is almost obvious.
\begin{lemma}\label{simple}
\begin{enumerate}
\item For a positive number $s$: $B(sg, sf, sF)=B(g,f,F)$. Moreover, if a pair $(\varphi, \psi)$ is admissible for a point $(g,f,F)$ then $(s\varphi, s\psi)$ is admissible for $(sg, sf, sF)$. If a pair $(\varphi, \psi)$ is an $\ep$-extremizer for a point $(g,f,F)$ then $(s\varphi, s\psi)$ is an $\ep$-extremizer for $(sg, sf, sF)$.
\item $B(g,f,F)=B(g, -f, F)$. Moreover, if a pair $(\varphi, \psi)$ is admissible for a point $(g,f,F)$ then $(-\varphi, \psi)$ is admissible for $(g, -f, F)$. If a pair $(\varphi, \psi)$ is an $\ep$-extremizer for a point $(g,f,F)$ then $(-\varphi, \psi)$ is an $\ep$-extremizer for $(g, -f, F)$.
\end{enumerate}
\end{lemma}

The next lemma is a key to our ``splitting'' technique.
\begin{lemma}\label{gluezero}
Suppose two pairs $(\varphi_\pm, \psi_\pm)$ are admissible for points $(g^\pm, f^\pm, F^\pm)$ correspondingly. Suppose also that 
$$
F=\frac{F^+ + F^-}2, \; \; \; f=\frac{f^+ + f^-}2,\;\;\;, g=\frac{g^+ + g^-}2,\;\;\;|f^+-f^-| = |g^+-g^-|.
$$
Then a pair $(\varphi, \psi)$ is admissible for the point $(g,f,F)$, where
$$
\varphi(x) = \begin{cases} \varphi_-(2x), &x\in [0, \frac12) \\
\varphi_+(2x-1), &x\in [\frac12, 1], \end{cases}
\;\;\;\;
\psi(x) = \begin{cases} \psi_-(2x), &x\in [0, \frac12) \\
\psi_+(2x-1), &x\in [\frac12, 1].\end{cases}
$$
\end{lemma}
\begin{proof}
It is clear that $\av{\varphi}{I_0} = f$, $\av{\psi}{I_0} = g$, and $\av{|\varphi|}{I_0} = F$. All we need to prove is that for any interval $I$ it is true that
$$
|(\psi, h_I)| = |(\varphi, h_I)|.
$$
For any interval $I\not= I_0$ it is obvious, since pairs $(\varphi^\pm, \psi_\pm)$ are admissible for corresponding points. Thus, we need to show that
$$
|(\varphi, h_{I_0})|=|(\psi, h_{I_0})|.
$$
But
$$
(\varphi, h_{I_0}) = \av{\varphi}{[\frac12,1]} - \av{\varphi}{[0, \frac12]} = \av{\varphi_+}{[0,1]}-\av{\varphi_-}{[0,1]} = f^+-f^-,
$$
$$
(\psi, h_{I_0}) = \av{\psi}{[\frac12,1]} - \av{\psi}{[0, \frac12]} = \av{\psi_+}{[0,1]}-\av{\psi_-}{[0,1]} = g^+-g^-,
$$
which finishes our proof.
\end{proof}
We generalize this lemma a little.
\begin{lemma}\label{glue}
Suppose two pairs $(\varphi_\pm, \psi_\pm)$ are admissible for points $(g^\pm, f^\pm, F^\pm)$ correspondingly. Suppose also that 
$$
F=\frac{F^+ + F^-}2, \; \; \; f=\frac{f^+ + f^-}2,\;\;\;, g=\frac{g^+ + g^-}2,\;\;\;|f^+-f^-| = |g^+-g^-|.
$$
Suppose $I$ is a dyadic interval with ``sons'' $I_\pm$. Suppose that a pair $(\Phi, \Psi)$ is admissible for some point $(g^0, f^0, F^0)$. Suppose that 
$$
\forall \; x\in I \; \; \Phi(x)=\varphi^I(x), \; \; \; \Psi(x)=\psi^I(x),
$$
where the pair $(\varphi, \psi)$ is admissible for the point $(g,f,F)$. Then the pair $(\Phi_1, \Psi_1)$, defined below, is admissible for the point $(g^0, f^0, F^0)$:
$$
\Phi_1(x) = \begin{cases} \Phi(x), &x\not\in I \\
													\varphi_+^{I_+}(x), &x\in I_+\\
													\varphi_-^{I_-}(x), &x\in I_-
						\end{cases}, \;\;\;\;
						\Psi_1(x) = \begin{cases} \Psi(x), &x\not\in I \\
													\psi_+^{I_+}(x), &x\in I_+\\
													\psi_-^{I_-}(x), &x\in I_-
						\end{cases}
$$
\end{lemma}
Essentially this lemma says that if we have pairs $(\varphi_\pm, \psi_\pm)$, and and a pair $(\varphi, \psi)$ defined in the Lemma \ref{gluezero}, then we can split this pair into $(\varphi^\pm, \psi^\pm)$, defined on $I^\pm$ correspondingly. The proof of the Lemma \ref{glue} is essentially the same as the proof of the Lemma \ref{gluezero}.
\subsection{Change of variables}
It will be more convenient for us to work in variables
$$
y_1 = \frac{f-g}2, \; \; \; y_2 = \frac{-f-g}2, \;\;\; F.
$$
We define $M(y_1, y_2, F)=B(g,f,F)$. Then all properties of $B$ are easily translated to properties of $M$. Moreover, the ``splitting'' lemmas \ref{gluezero}, \ref{glue} remain true for fixed $y_1$ or fixed $y_2$.

If we have a point $(y_1, y_2, F)$ then by $(\varphi_{(y_1, y_2, F)}, \psi_{(y_1, y_2, F)})$ we denote an admissible pair for this point. An individual function $\varphi_{(y_1, y_2, F)}$ is always such that there is a function $\psi_{(y_1, y_2, F)}$, such that the pair $(\varphi_{(y_1, y_2, F)}, \psi_{(y_1, y_2, F)})$ is admissible for $(y_1, y_2, F)$.

\subsection{The proof of $\B \geqslant B$}
We will work in the $y$-variables. In these variables it is true that the function $M$ is concave when $y_1$ or $y_2$ is fixed. This is proved in the Theorem \ref{tuda}. 
Analogously to the previous definition, we define
$$
\mathcal{M}(y_1,y_2,F)=\B(g,f,F).
$$
We first prove that 
$$
\M(1,1, F)\geqslant M(1,1,F).
$$
Fix a large integer $r$ and set $\delta=\frac1{2^r}$. We notice the following chain of inequalities:
\begin{multline}
\M(1,1,F)\geqslant \frac12 \left( \M(1,1-\d, F+\d(1-F)) + \M(1,1+\d, F-\d(1-F))\right) = \\=\frac12 \left( \M(1,1-\d, F+\d(1-F)) + \M(1+\d, 1, F-\d(1-F))\right).
\end{multline}
Applying the same concavity we see that
$$
\M(1,1-\d, F+\d(1-F)) \geqslant \delta \M(1,0,1) + (1-\delta)\M(1,1,F) = \delta + (1-\delta)\M(1,1,F).
$$
Moreover, by the concavity
\begin{multline}
\M(1+\d, 1, F-\d(1-F))\geqslant \\
(\d-\d^2)\M(1+\d, 0, 1+\d) + (1-\d) \M(1+\d, 1+\d, (1+\d)(F-\d(2-F))) + \d^2 \M(1+\d, 1, F-\d(1-F)) \geqslant \\ \d-\d^2 + (1-\d)\M(1,1,F-\d(2-F))
\end{multline}
Therefore, we get 
$$
\M(1,1,F) \geqslant \frac12 \left( \delta + (1-\d)\M(1,1,F) + \d-\d^2 + (1-\d)\M(1,1,F-\d(2-F))\right),
$$
or
$$
\M(1,1,F)\geqslant \frac{2\d-\d^2}{1+\d} + \frac{1-\d}{1+\d}\M(1,1,F-\d(2-F)).
$$
Notice that it is true for any $F$. 
We now denote
$$
F^k = 2-(2-F)(1+\d)^k.
$$
Then, clearly, $F^0=F$, and $F^{k+1} = F^k - \d(2-F^k)$.
With this notation we get
$$
\M(1,1,F)\geqslant \frac{2\d-\d^2}{1+\d}\sli_{k=0}^K \left(\frac{1-\d}{1+\d} \right)^k + \left(\frac{1-\d}{1+\d}\right)^{K+1}\cdot \M(1,1,F^{K+1}).
$$

\subsubsection{The case $F\geqslant 2$}
In this case we have $F^{k+1}\geqslant F^k$, and therefore the point $(1, 1+\d, F^k-\d(1-F^k))$ always lies in $\Omega$. Thus, we can take $K$ as huge as we want. Therefore,
$$
\M(1,1,F)\geqslant \frac{2\d-\d^2}{1+\d}\sli_{k=0}^\infty \left(\frac{1-\d}{1+\d} \right)^k = \frac{2\d - \d^2}{2\d}.
$$
This is true for arbitrary small $\d$, and thus $\M(1,1,F)\geqslant 1$.

\subsubsection{The case $F\leqslant 2$}
In this case to assure that $(1,1+\d, F^k-\d(1-F^k))\in \Omega$ we need $F^k-\d(1-F^k)\geqslant \d$, which implies
$$
(1+\d)^{K+1}\leqslant \frac{2}{2-F}.
$$
Take $K\in [ \frac{\log \frac{2}{2-F}}{\log(1+\delta)}-10, \frac{\log \frac{2}{2-F}}{\log(1+\delta)}+10]$, such that this inequality holds. Then we get 
$$
\M(1,1,F)\geqslant \frac{2\d-\d^2}{1+\d}\sli_{k=0}^K \left(\frac{1-\d}{1+\d} \right)^k = \frac{2\d-\d^2}{2\d} \left(1-\left(\frac{1-\d}{1+\d}\right)^{K+1}\right).
$$
It is only left to notice that with our choise of $K$ we have
$$
\left(\frac{1-\d}{1+\d}\right)^{K+1} \to \frac{(2-F)^2}{4}, \;\;\; \d\to 0,
$$
and therefore
$$
\M(1,1,F)\geqslant 1-\frac{(2-F)^2}{4}=M(1,1,F).
$$

We leave the proof of the general inequality $\M(y_1, y_2,F)\geqslant M(y_1, y_2,F)$ to the reader. In fact, it is a simple use of the concavity of $\M$ along the line that connects $(y_1,0,y_1)$ with $(y_1, y_2, F)$.

\subsection{Building the extremal sequense for points $(1, 1, F)$}
The aim of this Section is to prove that $B(g, f, F)\leqslant \B(g, f, F)$ by a construction of an extremal sequense of pairs $(\varphi_n, \psi_n)$.
For the sake of simplicity, we do it only for the case $f-g=2$.

Due to the homogeneity and symmetry of the function $B$ it is enough to prove that 
$$
B(g, f, F)\leqslant \B(g, f, F)
$$
for $f\geqslant 0$, $f-g=2$. In the new variables it means that we consider the case $y_1=1$, and $y_2\leqslant y_1=1$. As we have seen, for $f\geqslant -g$ we have $B(g, f, F)=\B(g, f, F)=1$, and so we need to consider the case $f\leqslant -g$, i.e. $y_2\geqslant 0$. 
We first build the $\ep$-extremizer for the point $(F, 1, 1)$.

Fix a large integer $r$ and let $\delta=2^{-r}$. As before, denote $I_0=[0,1]$. Also denote $J_i = [2^{-i}, 2^{-i+1})$, 
Denote $m_i(x)=2^ix-1$ --- the linear function from $J_k$ onto $I_0$. 

We need the following lemma.
\begin{lemma}\label{constrrr}
Suppose $\delta=2^{-r}$ is small enough. Also, fix a small number $\ep>0$. Suppose $F^{1}=F-\delta(2-F)$, and the pair $(\varphi_{(1,1,F^1)}, \psi_{(1,1,F^1)})$ is admissible. Then there exists an admissible pair $(\varphi_{(1,1,F)}, \psi_{(1,1,F)})$ such that
\begin{equation}\label{eqqqqq}
|\{x\colon \psi_{(1,1,F)}\geqslant 0\}| \geqslant \frac{2\delta -\delta^2}{1+\delta} + \frac{1-\delta}{1+\delta} |\{x\colon \psi_{(1,1,F^1)}\geqslant 0\}| - \ep.
\end{equation}
\end{lemma}
\begin{proof}
First, we explain our strategy. In what follows, we always assume that functions on the right-hand side are already defined. We specify their definition later; however, we clearly indicate points to which the functions are admissible. 

We define
$$
\varphi_{(1,1,F)}(x) = \begin{cases} \varphi_{(1, 1-\delta, F+\delta(1-F))}(m_1(x)), &x\in J_1\\
\varphi_{(1, 1+\delta, F-\delta(1-F))}(2x), &x\in [0, \frac12). \end{cases}
\;\;\;\;
\psi_{(1,1,F)}(x) = \begin{cases} \psi_{(1, 1-\delta, F+\delta(1-F))}(m_1(x)), &x\in J_1\\
\psi_{(1, 1+\delta, F-\delta(1-F))}(2x), &x\in [0, \frac12). \end{cases}
$$
This splitting is illustrated on the following picture.
\begin{center}
\includegraphics{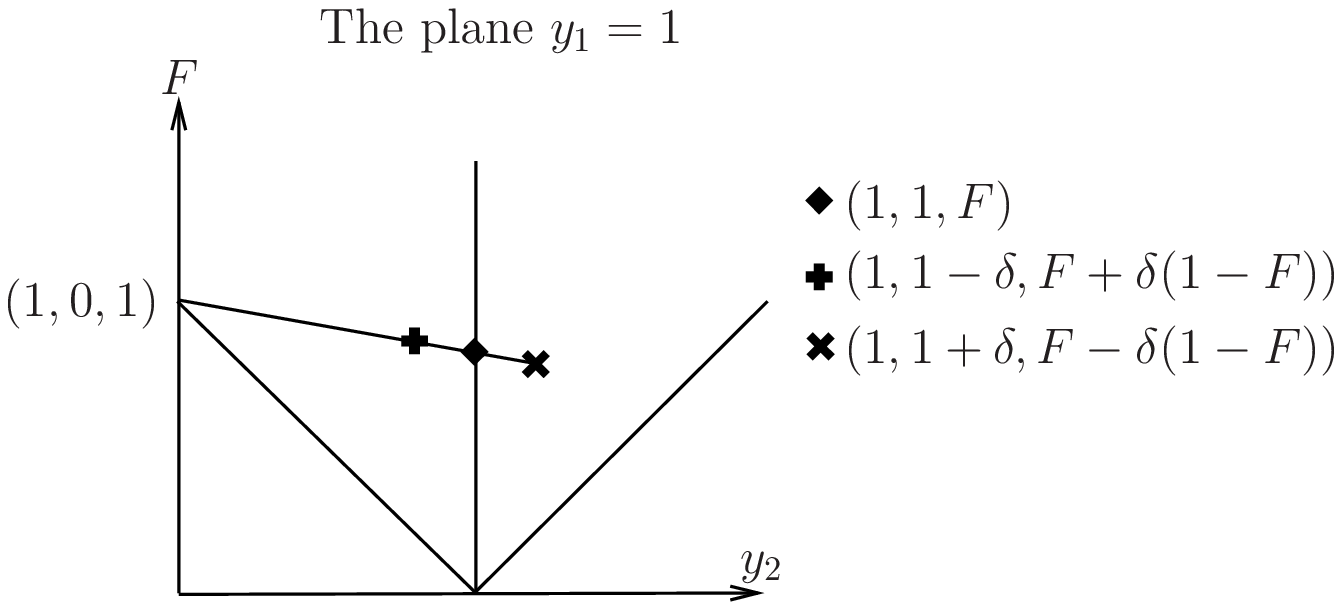}
\end{center}
By the Lemma \ref{glue} we see that $\psi_{(1,1,F)}$ is a martingale transform of $\varphi_{(1,1,F)}$. We define next
\begin{multline}
\varphi_{(1,1,F)}(x) = \begin{cases} \varphi_{(1, 0,1)}(\frac{m_1(x)}{\delta})), &x\in m_1^{-1}(\delta I_0)\\
\varphi_{(1,1,F)}(m_k(m_1(x))), &x\in m_1^{-1}m_k^{-1}(I_0), \; k=1\ldots r \\
-\varphi_{(1+\delta, 1, F-\delta(1-F))}(2x), &x\in [0, \frac12)
. \end{cases}
\;\;\;\; \\
\varphi_{(1,1,F)}(x) = \begin{cases} \psi_{(1, 0,1)}(\frac{m_1(x)}{\delta})), &x\in m_1^{-1}(\delta I_0)\\
\psi_{(1,1,F)}(m_k(m_1(x))), &x\in m_1^{-1}m_k^{-1}(I_0), \; k=1\ldots r \\
\psi_{(1+\delta, 1, F-\delta(1-F))}(2x), &x\in [0, \frac12)
. \end{cases}
\end{multline}

By the Lemma \ref{simple} and a multiple application of the Lemma \ref{glue}, we still get an admissible pair for the point $(1,1,F)$. 

Finally, define 

\begin{multline}
\varphi_{(1+\delta, 1, F-\delta(1-F))}(x) = \begin{cases} 
\varphi_{(1+\delta, 0, 1+\delta)}(m_k(\frac x\delta)), &x\in \delta \cdot J_k, \; k=1\ldots r\\
\varphi_{(1+\delta, 1, F-\delta(1-F))}(\frac{x}{\delta^2}), &x\in [0, \delta^2)\\
(1+\delta)\varphi_{(1, 1, F-\delta(2-F))}(m_k(x)), &x\in J_k, \; k=1\ldots r\\
\end{cases}
\\
\psi_{(1+\delta, 1, F-\delta(1-F))}(x) = \begin{cases} 
\psi_{(1+\delta, 0, 1+\delta)}(m_k(\frac x\delta)), &x\in \delta \cdot J_k, \; k=1\ldots r\\
\psi_{(1+\delta, 1, F-\delta(1-F))}(\frac{x}{\delta^2}), &x\in [0, \delta^2)\\
(1+\delta)\psi_{(1, 1, F-\delta(2-F))}	(m_k(x)), &x\in J_k, \; k=1\ldots r\\
\end{cases}
\end{multline}

This splitting is illustrated below.
\begin{center}
\includegraphics{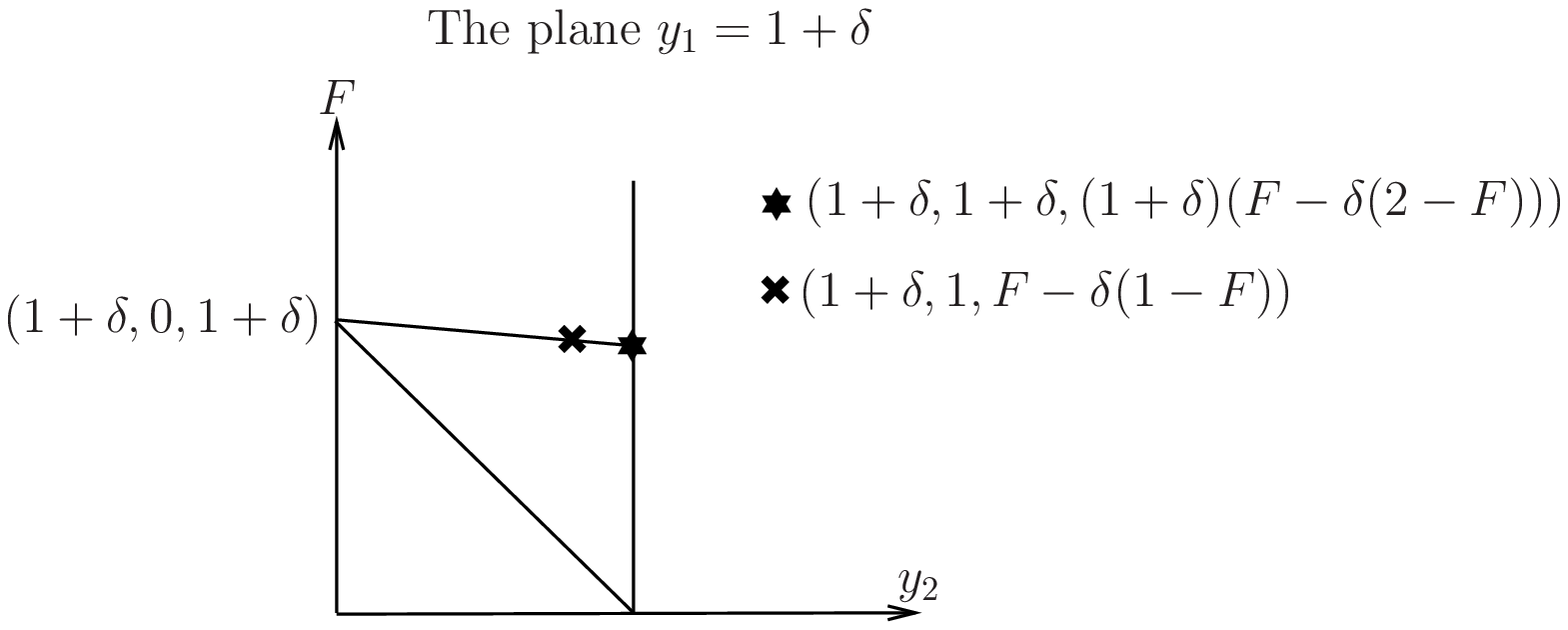}
\end{center}

Again, the Lemma \ref{simple} and the Lemma \ref{glue} assure that the defined pair is admissible.

Bringing everything together, we get

\begin{multline}\label{recccc}
\varphi_{(1,1,F)}(x) = \begin{cases} \varphi_{(1, 0,1)}(\frac{m_1(x)}{\delta})), &x\in m_1^{-1}(\delta I_0)\\
\varphi_{(1,1,F)}(m_k(m_1(x))), &x\in m_1^{-1}m_k^{-1}(I_0), \; k=1\ldots r \\
-\varphi_{(1+\delta, 0, 1+\delta)}(m_k(\frac{2x}{\delta})), &x\in \frac{\delta}{2}m_k^{-1}(I_0), \; k=1\ldots r\\
-\varphi_{(1+\delta, 1, F-\delta(1-F))}(\frac{2x}{\delta^2}), &x\in [0, \frac{\delta^2}{2})\\
-(1+\delta)\varphi_{(1, 1, F-\delta(2-F))}(m_k(2x)), &x\in \frac12 m_k^{-1}(I_0), \; k=1\ldots r
. \end{cases}
\;\;\;\; \\
\psi_{(1,1,F)}(x) = \begin{cases}  \psi_{(1, 0,1)}(\frac{m_1(x)}{\delta})), &x\in m_1^{-1}(\delta I_0)\\
\psi_{(1,1,F)}(m_k(m_1(x))), &x\in m_1^{-1}m_k^{-1}(I_0), \; k=1\ldots r \\
\psi_{(1+\delta, 0, 1+\delta)}(m_k(\frac{2x}{\delta})), &x\in \frac{\delta}{2}m_k^{-1}(I_0), \; k=1\ldots r\\
\psi_{(1+\delta, 1, F-\delta(1-F))}(\frac{2x}{\delta^2}), &x\in [0, \frac{\delta^2}{2})\\
(1+\delta)\psi_{(1, 1, F-\delta(2-F))}(m_k(2x)), &x\in \frac12 m_k^{-1}(I_0), \; k=1\ldots r
. \end{cases}
\end{multline}
We put two pictures together to show all five points involved in the splitting.
\begin{center}
  \includegraphics[width=.99\linewidth]{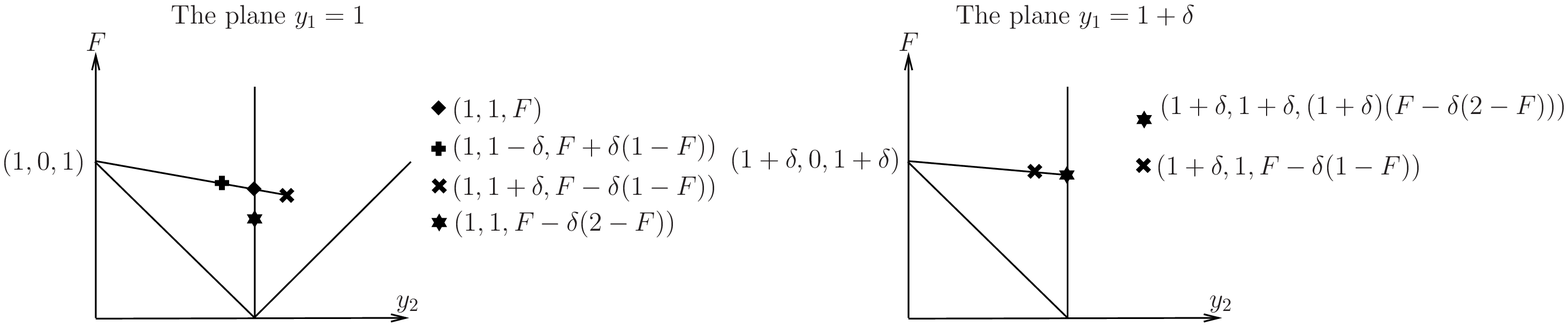}
 \end{center}
%

We now specify the definition of functions on the right-hand side. The pair $(\varphi_{(1,0,1)}, \psi_{(1,0,1)})$ is a $\frac\ep 2$-extremizer for the point $(1,0,1)$. The pair $(\varphi_{(1+\delta, 0, 1+\delta)}, \psi_{(1+\delta, 0, 1+\delta)})$ is a $\frac{\ep}{\delta-\delta^2}$-extremizer for the point $(1+\delta, 0, 1+\delta)$. 

The pair $(\varphi_{(1,1,F-\delta(2-F))}, \psi_{(1,1,F-\delta(2-F))})$ is given in the lemma. As for the pair $(\varphi_{(1+\delta, 1, F-\delta(1-F))}, \psi_{(1+\delta, 1, F-\delta(1-F))})$ --- we take any admissible pair for this point.

It is an easy calculation that the function $\psi_{(1,1,F)}$ satisfies the inequality \eqref{eqqqqq}. Moreover, it is easy to see that for {\bf any} pair, defined by \eqref{recccc} we have $\av{\varphi_{(1,1,F)}}{I_0} - \av{\psi_{(1,1,F)}}{I_0} = 2$. Thus, what we need to show is that there exists an admissible pair $(\varphi_{(1,1,F)}, \psi_{(1,1,F)})$ that satisfies the self-similarity condition \eqref{recccc}

To do that, we first take any admissible pair $(\tilde{\varphi}_{(1,1,F)}, \tilde{\psi}_{(1,1,F)})$ and define
\begin{multline}
\varphi^0_{(1,1,F)}(x) = \begin{cases} \varphi_{(1, 0,1)}(\frac{m_1(x)}{\delta})), &x\in m_1^{-1}(\delta I_0)\\
\tilde{\varphi}_{(1,1,F)}(m_k(m_1(x))), &x\in m_1^{-1}m_k^{-1}(I_0), \; k=1\ldots r \\
-\varphi_{(1+\delta, 0, 1+\delta)}(m_k(\frac{2x}{\delta})), &x\in \frac{\delta}{2}m_k^{-1}(I_0), \; k=1\ldots r\\
-\varphi_{(1+\delta, 1, F-\delta(1-F))}(\frac{2x}{\delta^2}), &x\in [0, \frac{\delta^2}{2})\\
-(1+\delta)\varphi_{(1, 1, F-\delta(2-F))}(m_k(2x)), &x\in \frac12 m_k^{-1}(I_0), \; k=1\ldots r
. \end{cases}
\;\;\;\; \\
\psi^0_{(1,1,F)}(x) = \begin{cases}  \psi_{(1, 0,1)}(\frac{m_1(x)}{\delta})), &x\in m_1^{-1}(\delta I_0)\\
\tilde{\psi}_{(1,1,F)}(m_k(m_1(x))), &x\in m_1^{-1}m_k^{-1}(I_0), \; k=1\ldots r \\
\psi_{(1+\delta, 0, 1+\delta)}(m_k(\frac{2x}{\delta})), &x\in \frac{\delta}{2}m_k^{-1}(I_0), \; k=1\ldots r\\
\psi_{(1+\delta, 1, F-\delta(1-F))}(\frac{2x}{\delta^2}), &x\in [0, \frac{\delta^2}{2})\\
(1+\delta)\psi_{(1, 1, F-\delta(2-F))}(m_k(2x)), &x\in \frac12 m_k^{-1}(I_0), \; k=1\ldots r
. \end{cases}
\end{multline}

Then the pair $(\varphi^0_{(1,1,F)}, \psi^0_{(1,1,F)})$ is admissible to point $(1,1,F)$. It is true by the Lemma \ref{glue}, and by an easy calculation that shows that all averages are as we need. 
We now define inductively
\begin{multline}
\varphi^{n+1}_{(1,1,F)}(x) = \begin{cases} \varphi_{(1, 0,1)}(\frac{m_1(x)}{\delta})), &x\in m_1^{-1}(\delta I_0)\\
\varphi^n_{(1,1,F)}(m_k(m_1(x))), &x\in m_1^{-1}m_k^{-1}(I_0), \; k=1\ldots r \\
-\varphi_{(1+\delta, 0, 1+\delta)}(m_k(\frac{2x}{\delta})), &x\in \frac{\delta}{2}m_k^{-1}(I_0), \; k=1\ldots r\\
-\varphi_{(1+\delta, 1, F-\delta(1-F))}(\frac{2x}{\delta^2}), &x\in [0, \frac{\delta^2}{2})\\
-(1+\delta)\varphi_{(1, 1, F-\delta(2-F))}(m_k(2x)), &x\in \frac12 m_k^{-1}(I_0), \; k=1\ldots r
. \end{cases}
\;\;\;\; \\
\psi^{n+1}_{(1,1,F)}(x) = \begin{cases}  \psi_{(1, 0,1)}(\frac{m_1(x)}{\delta})), &x\in m_1^{-1}(\delta I_0)\\
\psi^n_{(1,1,F)}(m_k(m_1(x))), &x\in m_1^{-1}m_k^{-1}(I_0), \; k=1\ldots r \\
\psi_{(1+\delta, 0, 1+\delta)}(m_k(\frac{2x}{\delta})), &x\in \frac{\delta}{2}m_k^{-1}(I_0), \; k=1\ldots r\\
\psi_{(1+\delta, 1, F-\delta(1-F))}(\frac{2x}{\delta^2}), &x\in [0, \frac{\delta^2}{2})\\
(1+\delta)\psi_{(1, 1, F-\delta(2-F))}(m_k(2x)), &x\in \frac12 m_k^{-1}(I_0), \; k=1\ldots r
. \end{cases}
\end{multline}

Then for any $n$ we get an admissible pair to the point $(1,1,F)$. 

We need to notice that
\begin{multline}
\ili_{I_0}|\varphi^{n+1}_{(1,1,F)} - \varphi^n_{(1,1,F)}|^2dx = \sli_k \frac{|J_k|}2 \ili_{I_0}|\varphi_{(1,1,F)}^{n}-\varphi_{(1,1,F)}^{n-1}|^2 dx = \frac{1-\delta}2 \ili_{I_0}|\varphi_{(1,1,F)}^{n}-\varphi_{(1,1,F)}^{n-1}|^2 dx = \\ = (\frac{1-\delta}2)^n \ili_{I_0}|\varphi_{(1,1,F)}^{1}-\varphi_{(1,1,F)}^{0}|^2 dx.
\end{multline}
Thus, we can take
$$
\varphi_{(1,1,F)} = \lim \varphi^{n+1}_{(1,1,F)} \;\;\; \mbox{in} \;\;\; L^2(I_0).
$$
Similarly 
$$
\psi_{(1,1,F)} = \lim \psi^{n+1}_{(1,1,F)} \;\;\; \mbox{in} \;\;\; L^2(I_0).
$$
It is clear that the pair $(\varphi_{(1,1,F)}, \psi_{(1,1,F)})$ satisfies the self-similarity conditions \eqref{recccc}. Moreover, since the limit in $L^2$ implies the limin in $L^1$, we get that all the averages are as needed. Moreover, for every interval $I$:
$$
|(\varphi_{(1,1,F)}, h_I)| = \lim |(\varphi^n_{(1,1,F)}, h_I)| = |(\psi^n_{(1,1,F)}, h_I)| = |(\psi_{(1,1,F)}, h_I)|,
$$
and thus we get an admissible pair. The proof of the lemma is finished.
\end{proof}

We are now ready to finish the whole construction. 
We consider a sequence 
$$
F^k = 2-(2-F)(1+\delta)^k.
$$
Then it is clear the $F^0 = F$ and $F^{k+1} = F^k - \delta(2-F^k)$. 
\subsubsection{The case $F\geqslant 2$}
We take a huge number $N$ and a small number $\ep$. Take any admissible pair $(\varphi_{(1,1,F^N)}, \psi_{(1,1,F^N)})$. Using the Lemma \ref{constrrr} $N$ times we build an admissible pair $(\varphi_{(1,1,F)}, \psi_{(1,1,N)})$. Moreover, we get
$$
|\{x\colon \psi_{(1,1,F)}(x)\geqslant 0\}| \geqslant \frac{2\delta - \delta^2}{1+\delta}\sli_{k=0}^N \left(\frac{1-\delta}{1+\delta}\right)^k - N\ep. 
$$
We now specify the choise of $\delta$, $N$ and $\ep$. We first fix a small $\delta$, so that $\frac{2\delta - \delta^2}{2\delta}=1-\sigma$. Then fix a huge number $N$, such that $\sli_{k=0}^N \left(\frac{1-\delta}{1+\delta}\right)^k > \frac{1+\delta}{2\delta}-\sigma\frac{1+\delta}{2\delta-\delta^2}$. Finally, fix a very small number $\ep$, such that $N\ep < \sigma$. Then we get

$$
|\{x\colon \psi_{(1,1,F)}(x)\geqslant 0\}| \geqslant \frac{2\delta - \delta^2}{1+\delta}\left(\frac{1+\delta}{2\delta}- \sigma\frac{1+\delta}{2\delta-\delta^2}\right) -\sigma= 1-3\sigma.
$$
where $\sigma$ is an arbitrary small number. 

\subsubsection{The case $F<2$}
We remind that our very first step requires that the point $(1, 1+\delta, F-\delta(1-F))$ to be in our domain. Thus, the on the $N$-th iteration we need that the point $(1, 1+\delta, F^N-\delta(1-F^N))$ is in the domain $\Om = \{(y_1, y_2, F)\colon F\geqslant |y_1-y_2| \}$. This yields to the inequality
$$
(1+\delta)^{N+1}<\frac{2}{2-F}.
$$
Thus, we should stop at the $K$-th step with
$$
(1+\delta)^{N+1}\approx\frac{2}{2-F}.
$$
Here the sign ``$\approx$'' means that 
$$
N\in [ \frac{\log \frac{2}{2-F}}{\log(1+\delta)}-10, \frac{\log \frac{2}{2-F}}{\log(1+\delta)}+10].
$$

We again apply the Lemma \ref{constrrr} $N$ times and get
$$
|\{x\colon \psi_{(1,1,F)}(x)\geqslant 0\}| \geqslant \frac{2\delta - \delta^2}{1+\delta}\sli_{k=0}^N \left(\frac{1-\delta}{1+\delta}\right)^k-N\ep=\frac{2\delta-\delta^2}{2\delta} \left(1- \left(\frac{1-\delta}{1+\delta} \right)^{N+1} \right) -N\ep
$$

Finally, since
$$
N\in [ \frac{\log \frac{2}{2-F}}{\log(1+\delta)}-10, \frac{\log \frac{2}{2-F}}{\log(1+\delta)}+10]
$$
we get that $\delta\to 0$ implies $1- \left(\frac{1-\delta}{1+\delta} \right)^{N+1} \to 1-\frac{(2-F)^2}4$, which finishes our proof.

\section{How to find the Bellman function $\mathcal{B}$}
In this section we explain how did we search for the function $\mathcal{B}$ and find it. We start with the following lemma.
\label{tuda} Let $x^\pm$ be two points in $\Omega$ such tat
$|f^+-f^-|=|g^+-g^-|$ and $x=\frac12(x^++x^-)$. Then
\begin{equation}
\label{mi} \mathcal{B}(x)-\frac{\mathcal{B}(x^+)+\mathcal{B}(x^-)}2\ge 0\,.
\end{equation}
\begin{proof}
Fix $x^\pm\in\Omega$, and let $(\varphi^\pm, \psi^\pm)$ be two pairs of functions giving the supremum
for $\B(x^+)$, $\B(x^-)$ respectively up to a small number $\eta>0$. Write
$$
\f^\pm=f^\pm+\!\!\!\!\sum_{I\subseteq I_0,\;I\in D}\!\!\!\!(\varphi, h_I) h\ci
I\,, \qquad \psi^\pm=g^\pm+\!\!\!\!\sum_{I\subseteq I_0,\;I\in D}
\!\!\!\!\ve\ci I (\varphi, h_I) h\ci I\,,
$$

Consider
$$
\f(t):=
\begin{cases}
\f^+(2t-1)\,,&\text{if }t\in [\frac12, 1]
\\
\f^-(2t)\,,&\text{if }t\in [0, \frac12).
\end{cases}
$$
and
$$
\psi(t):=
\begin{cases}
\psi^+(2t-1)\,,&\text{if }t\in [\frac12, 1]
\\
\psi^-(2t)\,,&\text{if }t\in [0, \frac12)
\end{cases}
$$
Since $|x^+_1-x^-_1|=|x^+_2-x^-_2|$, the function $\psi$ is a martingale
transform of $\f$, and the pair $(\varphi, \psi)$ is an admissible pair of the test
functions corresponding to the point $x$. Therefore,
\begin{align*}
\B(x)&\ge\frac1{|I_0|}\big|\{t\in I_0\colon \psi(t)\ge0\}\big|
\\
&=\frac1{2|I^+_0|}\big|\{t\in [\frac12, 1]\colon \psi(t)\ge0\}\big|+
\frac1{2|I^-_0|}\big|\{t\in [0, \frac12)\colon \psi(t)\ge0\}\big|
\\
&\ge\frac12\B(x^+)+\frac12\B(x^-)-2\eta.
\end{align*}
Since this inequality holds for an arbitrary small $\eta$, we can pass to the
limit $\eta\to0$, what gives us the required assertion.
\end{proof}
\begin{corollary}
The lemma means that if we change variables $f=y_1-y_2$,
$g	=-y_1-y_2$, and introduce a function $M(y_1, y_2, F):= B(g,f,F)$ defined in the
domain $G:=\{y=(y_1, y_2, F)\in\R^3\colon |y_1-y_2|\le F\}$, then we get that for each
fixed $y_2$, $M(F, y_1,\,\cdot)$ is concave and for each fixed $y_1$,
$M(F,\,\cdot\,,y_2)$ is concave.
\end{corollary}
\subsection{The boundary $F=y_1-y_2$}
We start with considering a boundary case $F=f$ or, in the $y$ variables, $F=y_1-y_2$. It means that we consider only non-negative functions $\varphi$. 
By the homogeneity of the function $M$ we need to find a function $S$ of variable $s=\frac{y_1}{y_2}$, such that 
\begin{equation}\label{ineq111}
\left(S(\frac{y_1}{y_2})\right)''_{y_1y_1} \leqslant 0, \; \mbox{and} \; \left(S(\frac{y_1}{y_2})\right)''_{y_2y_2}\leqslant 0.
\end{equation}
We notice that when $g\to 0$ we have $s\to -1$ and we must have $S\to 0$. Thus, we get a condition
\begin{equation}
S(s)\to 0, \;\;\; \mbox{as} \;\;\; s\to -1.
\end{equation}
Moreover, we have seen that if $f\geqslant -g$ then $\mathcal{B}(g, f, F)=1$. In particular, it holds when $f=-g$. Therefore, we have $M(y_1, -y_1, 0)=1$. This implies that
$$
S(s)\to 1	, \;\;\; \mbox{as} \;\;\; s\to -\infty.
$$
From inequalities \eqref{ineq111} we get that
$$
S''(s)\leqslant 0, \;\;\;, s^2 S''(s) +2S'(s)\leqslant 0, \;\;\; s\in (-\infty, -1]. 
$$
Make the second inequality an equation (we are looking for the {\bf best} nontrivial $S$). We get
$$
S(s)=c_1 + \frac{c_2}s.
$$
The boundary conditions imply that
$$
S(s)=1-\frac1s,
$$
and therefore
$$
M(y_1, y_2, y_1-y_2)=1-\frac{y_2}{y_1} = \frac{y_1-y_2}{y_1},
$$
or
$$
B(g, f,f)=\frac{2f}{f-g}.
$$
Thus, we get an answer
\begin{equation}\label{BC1}
M(y_1, y_2, y_1-y_2)=\begin{cases} 1, & y_2\leqslant 0\\
\frac{y_1-y_2}{y_1}, &y_2\geqslant 0,
\end{cases}
\end{equation}
or
$$
B(g,f,f) = \begin{cases} 1, &f\geqslant -g\\
\frac{2f}{f-g}, &f\leqslant -g. \end{cases}
$$
\subsection{The domain $\Omega$}
We remind the reader that for a fixed $y_1$ the function $M$ is concave in variables $(F, y_2)$. We also remind the symmetry condition, i.e.
$$
M(y_1, y_2, F)=M(y_2, y_1, F).
$$
Let us differentiate this equation in $y_2$ and set $y_2=y_1$. Then we get an equation:
$$
M_{y_1}(y_1, y_1, F)=M_{y_2}(y_1, y_1, F).
$$
Moreover, due to the symmetry it is enough to find $M$ for $y_2 \leqslant y_1$. As before, we saw that for $f\geqslant -g$ we have $B(g, f, F)=1$, i.e. 
\begin{equation}\label{BC2}
\mbox{for $y_2\leqslant 0$, we have $M(y_1, y_2, F)=1$.}
\end{equation}
Thus, it is enough to consider the case $0\leqslant y_2 \leqslant y_1$. 
Denote $\Omega_{y_1} = \{(y_2, F)\colon F\geqslant |y_2-y_1|\}$ --- the section of $\Omega$ for fixed $y_1$. 
We want to find $M$ satisfying concavity in this hyperplane--we are going to look for $M$ (and we will check later that it is concave) that solves Monge--Amp\`ere (MA) equation in   $\Omega_{y_1}$ with boundary conditions \eqref{BC1} and \eqref{BC2}.
 In   $\Omega_{y_1}$, there  is a point
 $P:=(0,y_1, y_1)$. Let us make a guess that the characteristics (and we know by Pogorelov's theorem that they form the foliation of  $\Omega_{y_1}$ by straight lines) of  our MA equation in  $\Omega_{y_1}$ form the fan of lines with common point $P=(y_1, y_1, 0)$. By Pogorelov's theorem we also know that there exists functions $t_1, t_2, t$ constant on characteristics such that
 \begin{equation}
 \label{M}
 M=t_1 F+ t_2 y_2 +t\,,
 \end{equation}
 such that $t_1=t_1(t; y_1), t_2=t_2(t; y_1)$ (we think that $y_1$ is a parameter), that
 \begin{equation}
 \label{M1}
 0=(t_1)'_t F+( t_2)'_t y_2 +1\,,
 \end{equation}
 that
 \begin{equation}
 \label{M2}
 t_1=\frac{\partial M(\cdot, y_2, F)}{\partial F}\,,\, t_1=\frac{\partial M(\cdot, y_2, F)}{\partial y_2}\,.
 \end{equation}
 Let us call  characteristics $L_t$. Extend one of them from $P$ till $y_2=y_1$. We recall another boundary condition:
 \begin{equation}\label{Nor}
 \text{If} \,\, y_2=y_1\Rightarrow \frac{\partial M}{\partial y_2}= \frac{\partial M}{\partial y_1}\,.
 \end{equation}
 Or if we denote the intersection of $L_t$ with $y_2=y_1$ by $(y_1, y_1, F(t))$ we get

  \begin{equation}
 \label{Nor1}
 t_2(t;y_1)= \frac{\partial M}{\partial y_1}(y_1, y_1, F(t))\,.
 \end{equation}

 We want to prove now that

  \begin{equation}
 \label{Nor2}
 \text{On the whole}\,\, L_t \,\,\text{we have}\,\,F(t) t_1 + 2y_1 t_2=0\,.
 \end{equation}
 In fact, our $M$ is  $0$ homogeneous. So everywhere $FM'_F + y_1M'_{y_1} + y_2M'_{y_2}=0$. Apply this to point $(y_1, y_1, F(t))$, where we can use \eqref{Nor1} and get
 $F(t)t_1+t_2 y_1 + t_2 y_1=0$, which is \eqref{Nor2} in one point. But then all entries are constants on $L_t$, therefore, \eqref{Nor2} follows.

 Now use our guess that $L_t$ fan from $P=(y_1, y_1, 0)$. Plug this coordinates into $0=(t_1)'_t F+( t_2)'_t y_2 +1$, which is \eqref{M1}.  Then we get the crucial (and trivial) ODE
 \begin{equation}
 \label{t1}
 t_1'(t)=-\frac{1}{y_1}\Rightarrow t_1(t)=-\frac{1}{y_1}t +C_1(y_1) \,.
 \end{equation}

Let boundary line $F= y_1-u$ corresponds to $t=t_0$. Then we use \eqref{M} and \eqref{BC1}:
$$
 (  - \frac{1}{y_1}t _0+C_1(y_1)  )(y_1-u) + t_2 u+t_0 = 1-\frac{u}{y_1}\,.
 $$
 Using \eqref{Nor2} we can plug $t_2$ expressed via $F(t)$. But by definition $F(t_0)=0$. So we get
 $$
   (  - \frac{1}{y_1}t _0+C_1(y_1)  )(y_1-u) +t_0=  1-\frac{u}{y_1}\,.
   $$
   Or
   $$
   C_1(y_1)y_1 - (t_0 +C_1(y_1)y_1) \frac{u}{y_1} = 1 -\frac{u}{y_1}\,.
   $$
   Varying $u$ we get $C_1(y_1)=\frac1{y_1}$, $t_0=0$.   Now from \eqref{t1} we get
\begin{equation}
 \label{t11}
  t_1(t)=\frac1{y_1}(1-t) \,.
 \end{equation}

   After that \eqref{M1} and \eqref{Nor2} become the system of two linear ``ODE"s on $F(t)$ and $t_2(t)$:

 \begin{equation}
 \label{Ft2}
 \begin{cases}
 -\frac1{y_1} F(t) +y_1 t_2'(t)    +1 =0\\
 2y_1t_2(t) + F(t)\frac1{y_1}(1-t) =0\,.
 \end{cases}
 \end{equation}

  We find $t_2= -\frac1{y_1}(1-t)t$.  We find the arbitrary constant for $t_2$ by noticing that the second equation of \eqref{Ft2}  at $t_0=0$ implies that $t_2(0)=0$ as $F(t_0)=F(0)=0$ by definition.

Hence \eqref{M1} becomes
\begin{equation}
\label{M11}
-\frac1{y_1} F +\frac1{y_1} (2t-1) y_2 +1=0\,.
\end{equation}
 Given $(y_1, y_2, F)\in \Omega_{y_1}\cap \{0\le y_2\le y_1\}$, we find $t$ from \eqref{M11} and plug it into \eqref{M}, in which we know already $t_(t)$ and $t_2(t)$. Namely, we know that

\begin{equation}
\label{M01}
M(y_1, y_2, F) =\frac1{y_1} F -\frac1{y_1} t(1-t)y_2 +t \,.
\end{equation}

 Plugging $t=\frac12\frac{F-(y_1-y_2)}{y_2}$ from \eqref{M11} into this equation we finally obtain
 \begin{equation}
 \label{FBFeq0}
 M(y_1, y_2, F)=1-\frac{(F-y_1-y_2)^2}{4y_1y_2}.
 \end{equation}
We notice that on the line $F=y_2+y_1$ we get $M = 1$. Thus, we get the following answer for $M$:
\begin{equation}
 \label{FBFeq}
 M(y_1, y_2, F)=\begin{cases} 1-\frac{(F-y_1-y_2)^2}{4y_1y_2}, &F\leqslant y_1+y_2 \\
 1, &F\geqslant y_1+y_2. \end{cases}
 \end{equation}
In our initial coordinates we get
$$
B(g, f, F)=\begin{cases} 1-\frac{(F+g)^2}{g^2-f^2}, &F\leqslant -g \\
 1, &F\geqslant -g. \end{cases}
$$
\end{document}